\documentclass[11pt,english,reqno]{amsart}
\usepackage[T1]{fontenc}
\usepackage[latin9]{inputenc}
\usepackage{a4wide}
\usepackage{verbatim} %paquete comentarios
\usepackage{amsthm}
\usepackage{amssymb}
\usepackage{setspace}
\usepackage{enumerate}
\usepackage{fancyhdr}
\usepackage{xcolor}

\usepackage[colorlinks=true]{hyperref}
\hypersetup{linkcolor=red,citecolor=blue,filecolor=dullmagenta,urlcolor=blue}

\numberwithin{equation}{section}
\theoremstyle{plain}
\newtheorem{thm}{Theorem}[section]
\newtheorem{lem}[thm]{Lemma}

\newtheorem{cor}[thm]{Corollary}

\theoremstyle{definition}
\newtheorem{defn}{Definition}[section]

\theoremstyle{remark}
\newtheorem{rem}{Remark}[thm]

\newcommand{\R}{{\mathbb{R}}} 

\newcommand{\I}{\mathcal{I}(t)}

\newcommand{\J}{\mathcal{J}(t)}
\newcommand{\Op}{(1-\partial_x^2)}

\newcommand{\IOp}{(1-\partial_x^2)^{-1}}
\newcommand{\Int}[1]{\int_{#1}}

\makeatother

\usepackage{babel}
\addto\shorthandsspanish{\spanishdeactivate{~<>}}

\addto\captionsenglish{}
\addto\captionsenglish{}
\addto\captionsenglish{}
\addto\captionsenglish{}
\addto\captionsenglish{}
\addto\captionsenglish{}
\addto\captionsenglish{}
\addto\captionsspanish{}
\addto\captionsspanish{}
\addto\captionsspanish{}
\addto\captionsspanish{}
\addto\captionsspanish{}
\addto\captionsspanish{}
\addto\captionsspanish{}

\newcommand{\sech}{\operatorname{sech}}

\begin{document}
	\title[Asymptotics for Improved Boussinesq]{Decay in the one dimensional generalized Improved Boussinesq equation}
	\author[Christopher Maul\'en]{Christopher Maul\'en}  % in alphabetical order
	\address{Departamento de Ingenier\'{\i}a Matem\'atica and Centro
de Modelamiento Matem\'atico (UMI 2807 CNRS), Universidad de Chile, Casilla
170 Correo 3, Santiago, Chile.}
	\email{cmaulen@dim.uchile.cl}
	\thanks{(Ch.M.) Partially funded by Chilean research grants FONDECYT 1150202 and CONICYT PFCHA/DOCTORADO NACIONAL/2016-21160593}
	
		\author[Claudio Mu\~noz]{Claudio Mu\~noz} 
		\address{CNRS and Departamento de Ingenier\'{\i}a Matem\'atica and Centro
de Modelamiento Matem\'atico (UMI 2807 CNRS), Universidad de Chile, Casilla
170 Correo 3, Santiago, Chile.}
\email{cmunoz@dim.uchile.cl}
		\thanks{(Cl.M.) Partially funded by Chilean research grants FONDECYT  1150202 and 1191412, project France-Chile ECOS-Sud C18E06 and CMM Conicyt PIA AFB170001. Part of this work was done while Cl.M. was visiting the CMLS at Ecole Polytechnique, France; and the Departamento de Matem\'aticas Aplicadas de Granada, UGR, Spain.}
	\begin{abstract}
		We consider the decay problem for the generalized improved (or regularized)  Boussinesq model with power type nonlinearity, a modification of the originally ill-posed shallow water waves model derived by Boussinesq. This equation has been extensively studied in the literature, describing plenty of interesting behavior, such as global existence in the space $H^1\times H^2$, existence of super luminal solitons, and lack of a standard stability method to describe perturbations of solitons. The associated decay problem has been studied by Liu, and more recently by Cho-Ozawa, showing scattering in weighted spaces provided the power of the nonlinearity $p$ is sufficiently large. In this paper we remove that condition on the power $p$ and prove decay to zero in terms of the energy space norm $L^2\times H^1$, for any $p>1$, in two almost complementary regimes:  (i) outside the light cone for all small, bounded in time  $H^1\times H^2$ solutions, and (ii) decay on compact sets of arbitrarily large bounded in time $H^1\times H^2$ solutions. The proof consists in finding two new virial type estimates, one for the exterior cone problem based in the energy of the solution, and a more subtle virial identity for the interior cone problem, based in a modification of the momentum. 
	\end{abstract}
	\maketitle

	\section{Introduction}
	\subsection{Setting}
	This paper is concerned with the so-called generalized \emph{Improved Boussinesq} equation (gIB) \cite{Poch,Chree}
	\begin{equation}\label{eq:IMBq}
	%\begin{align}
		%\begin{cases}\label{eq:IMBq}
			\partial_{t}^2u-\partial_{x}^2\partial_{t}^2u-\partial_{x}^2 u-\partial_{x}^2 (|u|^{p-1}u) =0,\quad (t,x)\in \R\times \R,
			%u(t=0,x)=\varphi(x),\ \ \partial_{t}u(t=0,x)=\psi(x),\  x\in \R.
		%\end{cases}
	%\end{align}
	\end{equation}
	where $u=u(t,x)$ is a real-valued function, and $p>1$. Sometimes referred as the Pochhammer-Chree equation \cite{Liu}, this model was first introduced by Pochhammer \cite{Poch} in its linear version in 1876, and in its complete nonlinear form by Chree \cite{Chree}, in 1886. It was derived as a model of the longitudinal vibration of an elastic rod, as well as a model of nonlinear waves in weakly dispersive media, and shallow water waves.
	
	\medskip
	
	The model gIB \eqref{eq:IMBq} shares plenty of similarities with the so called \emph{generalized good and bad Boussinesq} models \cite{Bous}
	\begin{equation}\label{Bad_Good}
	\partial_{t}^2u \pm \partial_{x}^4 u-\partial_{x}^2 u-\partial_{x}^2 (|u|^{p-1}u) =0,\quad (t,x)\in \R\times \R,
	\end{equation}
	Here the plus sign denotes the good Boussinesq system, which is locally and globally well-posed in standard Sobolev spaces \cite{Notes_linares,Kishimoto_WP_2012}, and the {\it minus sign} represents the ``bad'' equation originally derived by Boussinesq \cite{Bous}, which is strongly linearly ill-posed. Precisely, motivated by the similar order of magnitude of $\partial_x$ and $\partial_t$ in shallow water waves, the linearized gIB model \eqref{eq:IMBq} was discussed by Whitham \cite[p. 462]{Whitham}. By doing the ``Boussinesq trick'' (changing two $\partial_x$ by two $\partial_t$) in the bad Boussinesq equation, one arrives to \eqref{eq:IMBq} and ill-posedness is no longer present. This regularization process leads to gIB \eqref{eq:IMBq}, also known as the \emph{regularized Boussinesq} equation.
	
	\medskip
	
	Although \eqref{eq:IMBq} is no longer strongly ill-posed as bad Boussinesq \eqref{Bad_Good}, it is still shares some of its unpleasant behavior, but also some nice surprising properties. In order to explain this in detail, we write \eqref{eq:IMBq} as the system
	\begin{align}\label{IB}
		\hbox{(gIB)}~ \begin{cases}
			\partial_{t}u =\partial_{x}v\\
			\partial_{t}v  =(1-\partial_{x}^{2})^{-1}\partial_{x}(u+|u|^{p-1}u).
		\end{cases}
	\end{align}
	This $2\times 2$ system is Hamiltonian, but as far as we understand, not integrable. Its Hamiltonian character leads to the conservation of energy and momentum, given by
	\begin{align}
		H(u,v)&=\dfrac{1}{2} \int_\R \left( u^2+v^2+(\partial_x v)^2 \right) dx+\dfrac{1}{p+1}\int_\R \vert u\vert^{p+1}dx, \label{eq:energia} \\
		P(u,v)&=\Int{\R} (uv+\partial_x u\partial_x v)dx. \label{eq:momentum}
	\end{align}
Note in particular the complex character of energy and momentum for gIB: the energy is always nonnegative, and makes sense e.g. for $u\in L^2 \cap L^{p+1}$, and $v\in H^1$. On the other hand, the momentum needs even more regularity than expected, and it is only well-defined for $(u,v)\in H^1\times H^1$ (or $L^2\times H^2$). Given this lack of concordance, completely contrary to classical linear waves, understanding the well-posedness problem in gIB is far from trivial. Indeed, it turns out that $L^2\times H^1$ seems not well suited to have a well-defined energy,  so in this work we shall work in the proper subspace $H^1\times H^2$, for the reasons explained below. 

\medskip

The pioneering work by Liu \cite{Liu} showed local and global well-posedness for \eqref{eq:IMBq} for data $(u_0,v_0) \in H^s\times H^{s+1}$ and $s\geq 1$. In addition, the energy and momentum \eqref{eq:energia}-\eqref{eq:momentum} are conserved by the flow, or in other words, the $L^2\times H^1$ norm of the solution remains bounded in time. Note however that the $H^1\times H^2$ norm of the solution need not be globally bounded in time. The method employed by Liu was essentially based in the Sobolev inclusion $H^1$ into $L^\infty$ in one dimension, since no useful dispersive decay estimates are available for gIB. The fact that the solvability space differs from the energy space is a property standard in quasilinear models, and gIB has the flavor of a standard one. Consequently, we believe that this weakly ill-posed behavior in gIB is deeply motivated and inherited by the original strongly ill-posed bad Boussinesq equation \eqref{Bad_Good}. Additionally, Liu also showed blow up of negative energy solutions of \eqref{IB} but with focusing nonlinearities (minus sign in $|u|^{p-1}u$ instead of plus sign). Finally, the controllability problem for gIB in a finite interval \eqref{IB} has been recently studied by Cerpa and Cr\'epeau \cite{CC}.

\medskip

	However, the gIB system \eqref{IB} also enjoys some nice properties. Indeed, this model is also characterized by the existence of \emph{super-luminal} solitary waves, or just naively solitons, of the form
	\begin{equation}\label{eqn:soliton0}
(u,v)=(Q_c,-cQ_c)(x-ct-x_0), \quad x_0\in\R, \quad |c|>1.
	\end{equation}
	The super-luminal character is represented by the condition $|c|>1$ on the speed. Here, the scaled soliton is slightly different from generalized Korteweg-de Vries (gKdV): $Q_c(s)=(c^2-1)^{1/(p-1)}Q\left(\sqrt{ \frac{c^2-1}{c^2}} s\right)$, and 
	\begin{equation}\label{Soliton}
	Q(s)= \left( \frac{p+1}{2\cosh^2\left(\frac{(p-1)s}{2}\right)}\right)^{\frac1{p-1}}>0
	\end{equation}
	 is the soliton that solves $Q''-Q+Q^p=0$, $Q\in H^1(\R)$. Note that $Q_c$ must solve the modified elliptic equation
	\begin{equation}\label{eqn:soliton}
	c^2 Q_c'' -(c^2-1)Q_c +Q_c^{p}=0.
	\end{equation}
	Since the speed of solitons can be arbitrarily large, it clearly implies that \eqref{IB} \emph{possesses infinite speed} of propagation, a fact not present in standard wave-like equations. Note also that solitons with speeds $|c|\downarrow 1$ are small in $L^\infty\cap H^1$, but they do not decay to zero as time evolves, in any standard norm.
	
	\medskip
	
	In this paper, we are motivated by the decay problem of solutions to gIB \eqref{IB}. This interesting question has attracted the attention of several people before us. Liu \cite{Liu} showed decay of solutions  to \eqref{IB} obtained from initial data satisfying e.g. $(u_0,v_0)$ in $H^1\times H^2$, $u_0\in L^1$ and $(1-\partial_x^2)^{1/2}v_0 \in L^1$, all of them small enough. In particular, he showed that for $p>12$, 
	\[
\sup_{t\geq 0}\left( (1+t)^{\frac1{10}} \|u(t)\|_{L^\infty} + \|(u,v)(t)\|_{H^1\times H^2}\right) <+\infty.
	\]
	He also showed that $p$ can be taken greater than 8 if $s>\frac32$ and $(u_0,v_0)$ in $H^s\times H^{s+1}$. 
Next, in \cite{wang-chen}, Wang and Chen extended this result to higher dimensions.

	\medskip
	
	The exponent $p$ in \eqref{IB} was recently improved by Cho and Ozawa \cite{cho-ozawa}, who showed using modified scattering techniques that $p$ can be taken greater than $9/2$ if $u_0\in H^s$, $s>\frac85$. The solution global in this case satisfies $\|u(t)\|_{L^\infty}=O(t^{-2/5})$ as $t\to +\infty$. Additionally, the same authors showed that the asymptotics as $t\to+\infty$ cannot be the linear one if $1<p\leq 2$ and near zero frequencies vanish at infinity. Lowering the exponent $p$ for which there is decay seems a complicated problem, due to the quasilinear behavior of gIB.	
	  
	\subsection{Main results} In this paper, we are interested in the asymptotics of gIB solutions in the lower $p$ case, namely any possible $p>1$. Since there should be modified dynamics, we believe that we need different tools to attack this problem. 
	
	\medskip
	
		Our first result deals with the exterior light-cone decay problem. More precisely let $a,b>0$ be arbitrary positive numbers. We consider the interval depending on time
	\begin{align}
		I(t)= \big(-\infty,-(1+a)t \big)\cup \big( (1+b)t,\infty \big), \qquad t>0.	\label{eq:interval}
	\end{align}
	Our first result shows that, regardless the power $p>1$, any global solution $(u,v)$ to \eqref{IB} which is sufficiently small and regular must concentrate inside the light cone.
	
	\begin{thm}[Decay in exterior light cones]\label{TH1}
		Let $(u,v)\in C(\R,H^1\times H^2)$ be a global small solution of \eqref{IB} such that, for some $\epsilon(a,b)>0$ small, one has
		\begin{align}\label{smallness0}
			\sup_{t\in\R}\Vert (u(t),v(t))\Vert_{H^1\times H^2}<\epsilon.
		\end{align}
		 Then, for $I(t)$ as in \eqref{eq:interval}, there is strong decay to zero in the energy space:
		\begin{align}\label{decay0}
			\lim_{t\to\infty} \Vert (u(t),v(t))\Vert_{(L^2\times H^1)(I(t))}=0.
		\end{align}
		Additionally, one has the mild rate of decay for $|\sigma|>1$:
		\begin{align}\label{integrability0}
			\int_{2}^{\infty} \!\! \Int{\R} e^{-c_0\vert x+\sigma t \vert}(u^2+v^2+(\partial_x v)^2) dxdt \lesssim
			_{c_0} \epsilon^2.
		\end{align}
	\end{thm}
	
	\begin{rem}
	Note that Theorem \ref{TH1} is sharp, since it does not persist in the large data case. Indeed, solitons \eqref{eqn:soliton0} can be arbitrarily large and do not decay in the energy norm inside $I(t)$ as time tends to infinity. 
	\end{rem}
	
	\begin{rem}
	The smallness condition \eqref{smallness0} is needed in the proof to get a well-defined flow and good boundedness properties of the $L^\infty_{t,x}$ norm of $u$, and we do not know if it can be improved to $\Vert (u_0,v_0)\Vert_{L^2\times H^1}<\epsilon$ only. Note also that only the conditions $\Vert (u_0,v_0)\Vert_{L^2\times H^1}<\epsilon$ and $\sup_{t\in\R}\Vert (u(t),v(t))\Vert_{\dot H^1\times \dot H^2}<\epsilon$ are essentially needed in the proofs here.
	\end{rem}
	
	The proof Theorem \ref{TH1} follows the introduction of a new virial identity, in the spirit of the previous results by Martel and Merle \cite{MM1,MM2} in the gKdV case, and \cite{munoz-kwak,ACKM} in the BBM case. Note however that in those cases the functional involved is related to the mass ($L^2$ norm) of the solution. Here, we use instead a modification of the energy \eqref{eq:energia} of the solution.
	
	\medskip
	
	Having described the small data behavior in exterior light cones, we concentrate now in the interior light cone behavior. Here things are much more complicated, since the energy \eqref{eq:energia} is no more useful to describe the dynamics. Instead, we shall use a suitable modification of the momentum \eqref{eq:momentum}. 
	
	\begin{thm}[Full decay in interior regions]\label{TH2}
		Let $(u,v)$ be a global solution of \eqref{IB} in  the class $ C(\R,H^1\times H^2) \cap L^\infty(\R,H^1\times H^2)$, not necessarily small in norm. Then for any $L\gg 1$ we have
		\begin{equation}\label{TH2:integrability}
			\int_{2}^{\infty} \int_{-L}^L\left(v^2 + u (1-\partial_x^2)^{-1}u+ |u|^{p+1}\right) (t,x)dxdt \lesssim 1.
		\end{equation}
		Moreover, we have strong decay to zero in the energy space $(L^2\times H^1)(I)$, for any $I$ bounded interval in space:
		\begin{equation}\label{TH2:decay}
			\lim_{t\to \infty} \|(u,v)(t)\|_{(L^2\times H^1)(I)}=0.
		\end{equation}
	\end{thm}
	
	\begin{rem}
	Estimate \eqref{TH2:integrability} shows that the local $L^2$ norm of $v$ is integrable in time, and some mixed norms of $u$. Note however that $u$ seems not locally $L^2$ integrable in time. However, \eqref{TH2:decay} shows that this norm indeed decays to zero in time (even if it is not integrable in time). 
	\end{rem}
	
	\begin{rem}
	The fact that explicit higher regularity is needed in Theorem \ref{TH2} is certainly a consequence of a sort of quasilinear behavior of gIB. See also \cite{AM} for a similar behavior in a fully quasilinear model, the 1+1 dimensional Born-Infeld. In some sense, although gIB is well-posed, still shares some bad behavior coming from the originally ill-posed Bad Boussinesq model \eqref{Bad_Good}. Note that this phenomenon does not occur in the Good Boussinesq case \cite{MPP}. Also, Theorem \ref{TH2} can be read as ``boundedness in time in $H^1\times H^2$ implies $L^2\times H^1$ time decay in compact sets of space''.
	\end{rem}
	
	\begin{rem}
	Note that arbitrary size solitons \eqref{eqn:soliton0}-\eqref{Soliton} satisfy the hypotheses in Theorem \ref{TH2}. Hence, \eqref{TH2:decay} it is also true for large solutions.
	\end{rem}
		
	The techniques that we use to prove Theorem \ref{TH2} are not new, and have been used to show decay for the Born-Infeld equation \cite{AM}, the good Boussinesq system \cite{MPP}, the Benjamin-Bona-Mahony (BBM) equation \cite{munoz-kwak}, and more recently in the more complex $abcd$ Boussinesq system \cite{KMPP,KM}. In all these works, suitable virial functionals were constructed to show decay to zero in compact/not compact regions of space. 
	
	\medskip
	
	However, the case of gIB is different by several reasons: first of all, the small data long time dynamics in all the aforementioned models is \emph{not singular}, in the sense that using well-cooked virial identities, one always gets integrability in time of the whole associated energy norm. This is a nice property present in plenty of Hamiltonian models so far. The gIB case is different because this last property is not true at all: we only gets integrability in time of very particular portions of the $L^2\times H^1$ norm (see \eqref{TH2:integrability}). This fact complicates matters, since proving \eqref{TH2:decay} will require to prove additional estimates, not coming from the virial itself, but instead coming from tricky bounds and preservation of sign conditions under the nonlocal operator $(1-\partial_x^2)^{-1}$, namely the maximum principle. Second, finding the right virial identity for gIB was a very complicate process, since no clear notion of decay is shown by computing variations of energy and momentum. For instance, the derivation of a localized version of the momentum law (see \eqref{eq:J})
	\[
	 \J  = \Int{\R} \varphi\left(\dfrac{x}{L}\right)\left( uv+ \partial_x u \partial_x v\right) (t,x)dx, \qquad L\gg1 ,
	\]
	leads to the badly behaved identity (see \eqref{eq:Jp})
	\[
		\begin{aligned}
			\dfrac{d}{dt}\J  =  & ~ \frac{1}{2L}\Int{\R} \varphi'\left( u^2+\frac{2}{p+1}|u|^{p+1}-v^2- (\partial_x v)^2\right)dx \\
			& ~{} - \frac1L\Int{\R}\varphi'u(1-\partial^2_x)^{-1} (u+|u|^{p-1}u)dx.
	\end{aligned}
	\]
	No evidence of good sign conditions is clearly shown here. This identity, valid only for the gIB case, is far from being useful (actually, it is the first case among the above mentioned equations where it fails to give decay information). The key to prove decay is an additional term in the virial, called $\mathcal N(t)$ (see \eqref{N}-\eqref{eq:Np}), that allows us to recover the positivity of the virial in \eqref{eq:lyapunov_functional}:
	\[
		\begin{aligned}
			-\dfrac{d}{dt}(\mathcal{J}(t) +\mathcal{N}(t)) = &~ \frac 1{2L} \Int{\R} \varphi' v^2 dx+ \frac1{2L} \Int{\R} \varphi' u (1-\partial^2_x)^{-1} u dx\\
			& ~{}  +  
				\frac{p-1}{(p+1)L} \Int{\R} \varphi' |u|^{p+1} dx  + 	\frac1{2L}  \Int{\R} \varphi' u(1-\partial^2_x)^{-1} (|u|^{p-1}u)dx,
		\end{aligned}
	\]
	 and therefore the decay property. In that sense, we believe that estimate \eqref{eq:lyapunov_functional} is a true finding, since it is the only positivity property that we have found so far in gIB around compact sets. Note also that at this point we are not able to fully recover the decay of the $L^2$ norm of $u$ locally in space, but instead only a portion of it, expressed in terms of the estimate $\int_1^\infty \Int{\R} \varphi' u (1-\partial^2_x)^{-1} u (t,x) \,dx dt<+\infty$. In order to prove the more demanding decay property \eqref{TH2:decay}, we need additional estimates where the hypothesis $(u,v) \in L^\infty(\R,H^1\times H^2)$ seems to be essential. He have no direct clue about whether or not this condition is also necessary, but it seems to appear in several quasilinear models \cite{AM}. 
%	 In any case, we are still able to prove here that mild growing in time data control implies sequential decay.
%	 
%	 \begin{cor}[Sequential decay for growing in time $H^1\times H^2$ data]\label{Corolario_tiempo}
%	 Assume now in Theorem \ref{TH2} that only $\limsup_{t\to +\infty} t^{-1}\|(u,v)(t)\|_{H^1\times H^2} <+\infty$. Then there is a sequence $t_n\uparrow+\infty$ for which $\lim_{n\to+\infty}\|(u,v)(t_n)\|_{(L^2\times H^1)(I)}=0$, $I$ bounded interval.
%	 \end{cor}

\subsection{More about solitons} One important question left open in this paper is the stability/instability of solitons \eqref{eqn:soliton0}-\eqref{Soliton}. But, as we shall explain below, this question is far from being trivial. However, we believe that part of the techniques introduced in this work could be useful to show a certain degree of (asymptotic) stability of the IB soliton.

\medskip

	Let us be more precise. In an influential work, Grillakis, Shatah, Strauss \cite{GSS} (GSS) obtained  sharp conditions for the orbital stability/instability of ground state solutions for a class of abstract Hamiltonian systems. This result was extended to another class of Hamiltonians of KdV type by Bona, Souganidis y Strauss \cite{BSS}.
	Hamiltonian systems as the ones considered in \cite{GSS} allow to introduce the Lyapunov functional $F:=H-cI$, where $H$ is the Hamiltonian and $I$ is functional generated by the translation invariance of the equation (usually, mass or momentum). Here, $c$ is the corresponding speed of the solitary wave. The stability of the solitary wave is then reduced to the understanding of the second variation of $F$, in the sense that $\partial^2 F>0$ leads to stability. Also, if the former positive condition is not satisfied, but the corresponding nonpositive manifold is spanned by two elements (directions) which are associated to the two degrees of freedom of the solitary waves (scaling and shifts), it is still is possible prove stability using $\partial^2 F$, but it is also necessary to restrict the class of perturbations to those which are orthogonal to the nonpositive directions.

	\medskip
	
	Smereka in \cite{Smereka} studied the soliton of IB \eqref{eq:IMBq} and observed that this soliton fits into the class of abstract Hamiltonian system studied by GSS. However, it is not possible to apply the GSS method since an important hypothesis is not satisfied. In fact, he observed that $\partial^2 F$ is nonpositive on an infinite number of directions, where two of them can be associated to the point spectrum, and the remaining with the continuous spectrum. Therefore, GSS is useless in this case. However, the same author showed numerical evidence that  if $dI(Q_c)/dc<0$, then the solitary waves are stable, and if $dI(Q_c)/dc>0$ the solitary waves seems to be unstable.
	
\medskip

	In a very important paper, Pego and Weinstein \cite{PW} proved (among other things) that $Q_c$ is linearly exponentially unstable in $H^1$ when
	\[
	1<c^2<\left(\dfrac{3(p-1)}{4+2(p-1)}  \right)^2,\quad \mbox{with }\  p>5.
	\]
	Their method combines the use of the Evans function as well as ODE techniques. They also showed \cite{PW_2} that the linear equation around $Q_c$ for $c\sim 1$ satisfies a {\bf convective stability} property, based in the similarity of IB with KdV for small speeds. This result has been successfully adapted to a more general setting by Mizumachi in a series of works \cite{Mizu1,Mizu2}. Whether or not the asymptotic stability results \`a la Martel-Merle \cite{MM1,MM2} can be applied to this case, is a challenging problem. An interesting result in this direction can be found in the recent work \cite{KMM_3}.

	\subsection*{Organization of this paper}	 This paper is organized as follows: Section \ref{Sect:2} deals with two new virial identities introduced in this paper. Section \ref{Sect:4} is devoted to the proof of Theorem \ref{TH1}. In Section \ref{Sect:5} we prove Theorem \ref{TH2}.

\subsection*{Acknowledgments} We thank C. Kwak and M. A. Alejo for several important remarks that helped to improve a first draft of this paper.

	\bigskip

	\section{Virial identities}\label{Sect:2}
	
	In this section we present two new virial identities for the gIB equation \eqref{IB}. One is related with the exterior light cone behavior (Theorem \ref{TH1}), and the other is useful for understanding the compact in space region (Theorem \ref{TH2}). 
	
	\medskip
	
	Let $L>0$ be a large parameter, and $\varphi=\varphi(x)$ be a smooth, bounded weight function, to be chosen later. For each $t,\sigma\in \R$ and $L>0$, we consider the following functionals.
	\begin{align}
	\mathcal I(t;L,\sigma)= \I =& \dfrac{1}{2}\Int{\R} \varphi\left(\dfrac{x+\sigma t}{L}\right)\left( u^2+v^2+(\partial_xv)^2+\frac{2}{p+1} |u|^{p+1}\right) (t,x)dx \label{eq:I}, \\
		%\hbox{and} &   \nonumber\\
		\mathcal J(t;L,\sigma)= \J  =& \Int{\R} \varphi\left(\dfrac{x+\sigma t}{L}\right)\left( uv+ \partial_x u  \partial_x v\right) (t,x)dx \label{eq:J}.
	\end{align}
Note that both functionals are generalizations of the energy and momentum introduced in \eqref{eq:energia}-\eqref{eq:momentum}, and they well-defined if $(u,v)\in H^1\times H^2$. This fact is essential for the proofs, and it is the key ingredient in both Theorems \ref{TH1}-\ref{TH2}. The following result describes the time variation of both functionals. %Each functional well-defined for $H^1$ functions.
	\begin{lem}[Energy and Momentum local variations]\label{lem:I'}
		For any $t\in \R$, one has
		\begin{align}
			\dfrac{d}{dt}\I =&~{} \dfrac{\sigma}{2L}\Int{\R} \varphi'\left( u^2+v^2+(\partial_x v)^2+\frac{2}{p+1} |u|^{p+1}\right) dx  \nonumber\\
			&  {} -\frac{1}{L}\Int{\R} \varphi'v\IOp(u+|u|^{p-1}u) dx, \label{eq:I'}
			\end{align}
			and
			\begin{align}
			\dfrac{d}{dt}\J =&~{} \frac{\sigma}{L}\Int{\R} \varphi'\left( uv+\partial_x u \partial_x v\right) dx  +\frac{1}{2L}\Int{\R} \varphi'\left( u^2+\frac{2}{p+1}|u|^{p+1}-v^2-(\partial_x v)^2\right)dx \label{eq:J'}\\
			&-\frac{1}{L}
			\Int{\R}\varphi'u(1-\partial^2_x)^{-1} (u+|u|^{p-1}u)dx. \nonumber
		\end{align}
	\end{lem}

	\begin{proof}[Proof of Lemma \ref{lem:I'}] We have two identities to prove.

	\medskip
		{\it Proof of \eqref{eq:I'}}. We compute using \eqref{IB}:
		\begin{align*}
			\dfrac{d}{dt}\I =&~{} \dfrac{\sigma}{2L}\Int{\R} \varphi'\left( u^2+v^2+(\partial_x v)^2+\frac{2}{p+1} |u|^{p+1}\right) dx \\
			&  + \underbrace{\Int{\R} \varphi \left( u\partial_t u+v\partial_t v+\partial_x v\partial_{xt}^2 v+ |u|^{p-1}u\partial_t u\right) dx}_{\mathcal{I}_1(t)}.
		\end{align*}
		We deal first with the term $\mathcal I_1$. We have from \eqref{IB}
		\begin{align*}
			\mathcal{I}_1(t) =&\Int{\R} \varphi \left( u\partial_t u+v\partial_t v+\partial_x v\partial_{xt}^2 v+ |u|^{p-1}u\partial_t u\right) dx\\
			=& \Int{\R} \varphi \left( u+\partial_{xt}^2 v+ |u|^{p-1}u\right)\partial_t u dx+ \Int{\R} \varphi v\partial_t v dx\\
			=& \Int{\R} \varphi \partial_t u\IOp(u+|u|^{p-1}u) dx+ \Int{\R} \varphi v\IOp(u+|u|^{p-1}u)_{x} dx\\
			=& \Int{\R} \varphi \partial_t u\IOp(u+|u|^{p-1}u) dx- \Int{\R} \partial_x (\varphi v)\IOp(u+|u|^{p-1}u) dx\\
			=& \Int{\R} \left(\varphi \partial_t u-\dfrac{\varphi'v}{L}-\varphi \partial_x v\right)\IOp(u+|u|^{p-1}u) dx\\
			=& -\frac{1}{L}\Int{\R} \varphi'v\IOp(u+|u|^{p-1}u) dx.
		\end{align*}
		Finally, using this last identity, and replacing in the derivative of $\mathcal I(t)$, we obtain 
		\begin{align*}
			\dfrac{d}{dt}\I &= \dfrac{\sigma}{2L}\Int{\R} \varphi'\left( u^2+v^2+(\partial_x v)^2+\frac{2}{p+1} |u|^{p+1}\right) dx \\
			&~{} -\frac{1}{L}\Int{\R} \varphi'v\IOp(u+|u|^{p-1}u) dx,
		\end{align*}
		as desired. 
		
		\medskip
		
		{\it Proof of \eqref{eq:J'}.} The proof here is similar to the previous one. We have
		\begin{align*}
			\frac{d}{dt}\J&= \frac{\sigma}{L}\Int{\R} \varphi'\left( uv+\partial_x u \partial_x v\right) dx +\Int{\R} \varphi\left( \partial_t uv+u\partial_t v+\partial_{xt} u \partial_x v+\partial_x u\partial_{xt}^2 v\right) dx \\
			&= \frac{\sigma}{L}\Int{\R} \varphi'\left( uv+\partial_x u \partial_x v\right) dx +\Int{\R} \varphi\left( \partial_x vv+u\partial_t v+\partial_x^2 v \partial_x v+\partial_x u\partial_{xt}^2 v\right) dx \\
			&= \frac{\sigma}{L}\Int{\R} \varphi'\left( uv+\partial_x u \partial_x v\right) dx +\frac{1}{2}\Int{\R} \varphi \partial_x \left( v^2+(\partial_x v)^2\right) dx\\
			&~{} \quad +\Int{\R} \varphi \left(u\partial_t v+\partial_x u\partial_{xt}^2 v\right) dx \\
			&=\frac{\sigma}{L}\Int{\R} \varphi'\left( uv+\partial_x u \partial_x v\right) dx -\frac{1}{2L}\Int{\R} \varphi'\left( v^2+(\partial_x v)^2\right)dx \\
			&~{} \quad  +\underbrace{\Int{\R} \varphi \left(u\partial_t v+\partial_x u\partial_{xt}^2 v\right) dx.}_{\mathcal{J}_2(t)} \\
		\end{align*}
		Now, integrating by parts,
		\begin{align*}
			\mathcal{J}_2(t)
			&= \Int{\R} \varphi u\partial_t v dx-\Int{\R} \partial_x (\varphi \partial_{xt}^2 v)  u dx=  \Int{\R} \varphi u\partial_t v dx- \frac{1}{L}\Int{\R}\varphi' \partial_{xt}^2 vudx-\Int{\R}\varphi \partial_{txx}^3 v  u dx\\
			&=  \Int{\R} \varphi u(1-\partial_x^2)\partial_t v dx- \frac{1}{L}\Int{\R}\varphi' \partial_{xt}^2 vudx\\
			&=  \Int{\R} \varphi u(1-\partial_x^2)(1-\partial_x^2)^{-1} \partial_x (u+|u|^{p-1}u) dx- \frac{1}{L}\Int{\R}\varphi' \partial_{xt}^2 vudx\\
			&=  \Int{\R} \varphi u \partial_x (u+|u|^{p-1}u) dx- \frac{1}{L}\Int{\R}\varphi' \partial_{xt}^2 vudx\\
			&=  \Int{\R} \varphi \partial_x \left(\frac{1}{2}u^2+\frac{p}{p+1}|u|^{p+1}\right) dx- \frac{1}{L}\Int{\R}\varphi'u \partial_{xt}^2 vdx\\
			&=  -\frac{1}{L}\Int{\R} \varphi' \left(\frac{1}{2}u^2+\frac{p}{p+1}|u|^{p+1}\right) dx  - \frac{1}{L}\Int{\R}\varphi'u (1-\partial^2_x)^{-1}\partial_x^2 (u+|u|^{p-1}u)dx\\
			&=  -\frac{1}{L}\Int{\R} \varphi' \left(\frac{1}{2}u^2+\frac{p}{p+1}|u|^{p+1}\right) dx   -\frac{1}{L} \underbrace{\Int{\R} \partial_x^2 (\varphi'u) (1-\partial^2_x)^{-1} (u+|u|^{p-1}u)dx.}_{\mathcal{J}_3(t)}\\
		\end{align*}
		We consider now the term $\mathcal J_3(t)$:
		\begin{align*}
			\mathcal{J}_3(t)
			%	&= \Int{\R}(\varphi'u)_{xx} (1-\partial^2_x)^{-1} (u+u^p)dx\\
			&= \Int{\R}((\varphi'u)_{xx}-\varphi'u +\varphi'u)(1-\partial^2_x)^{-1} (u+|u|^{p-1}u)dx\\
			&=  -\Int{\R}(1-\partial^2_{x})(\varphi'u)(1-\partial^2_x)^{-1} (u+|u|^{p-1}u)dx   +
			\Int{\R}\varphi'u(1-\partial^2_x)^{-1} (u+|u|^{p-1}u)dx\\
			&=  -\Int{\R} \varphi'(u^2+|u|^{p+1})dx+
			\Int{\R}\varphi'u(1-\partial^2_x)^{-1} (u+|u|^{p-1}u)dx.
		\end{align*}
		Therefore,
		\begin{align*}
			\frac{d}{dt}\J & =  \frac{\sigma}{L}\Int{\R} \varphi'\left( uv+\partial_x u \partial_x v\right) dx -\frac{1}{2L}\Int{\R} \varphi'\left( v^2+(\partial_x v)^2\right)dx \\
			&   -\frac{1}{L}\Int{\R} \varphi' \left(\frac{1}{2}u^2+\frac{p}{p+1}|u|^{p+1}\right) dx  +\frac{1}{L} \Int{\R} \varphi'(u^2+|u|^{p+1})dx\\
			&- \frac{1}{L}\Int{\R}\varphi'u(1-\partial^2_x)^{-1} (u+|u|^{p-1}u)dx \\
			=&  \frac{\sigma}{L}\Int{\R} \varphi'\left( uv+\partial_x u \partial_x v\right) dx -\frac{1}{2L}\Int{\R} \varphi'\left( v^2+(\partial_x v)^2\right)dx  \\
			& +\frac{1}{L}\Int{\R} \varphi' \left(\frac{1}{2}u^2+\frac{1}{p+1}|u|^{p+1}\right) dx -\frac{1}{L}
			\Int{\R}\varphi'u(1-\partial^2_x)^{-1} (u+|u|^{p-1}u)dx\\
			=&  \frac{\sigma}{L}\Int{\R} \varphi'\left( uv+\partial_x u \partial_x v\right) dx   +\frac{1}{2L}\Int{\R} \varphi'\left( u^2+\frac{2}{p+1}|u|^{p+1}-v^2-(\partial_x v)^2\right)dx\\
			&  -\frac{1}{L}
			\Int{\R}\varphi'u(1-\partial^2_x)^{-1} (u+|u|^{p-1}u)dx.
		\end{align*}
	\end{proof}  
	
	We introduce now a second functional. Let
	
	\begin{equation}\label{N}
		\mathcal{N}(t) := \frac1{2L}\Int{\R} \varphi' \left(\frac{x}{L}\right) u \partial_x v dx.
	\end{equation}
	
	\begin{lem}\label{lem:N}
		We have for $ \varphi' = \varphi' \left(\frac{x}{L}\right) $,
		\begin{equation}\label{eq:Np}
			\begin{aligned}
				\dfrac{d}{dt}\mathcal{N}(t) = &~   \frac1{2L}\Int{\R} \varphi'  \Big( (\partial_x v)^2 -  u^2 + u(1-\partial^2_x)^{-1}u \Big)dx\\
				&~{}    + \frac1{2L}\Int{\R} \varphi'  \Big(  - |u|^{p+1} +  u(1-\partial^2_x)^{-1}(|u|^{p-1}u) \Big)dx.
			\end{aligned}
		\end{equation}
	\end{lem}
	
	\begin{proof}
		We compute using \eqref{IB}
		\begin{align*}
			 2L\dfrac{d}{dt}\mathcal{N}(t) &=  \dfrac{d}{dt} \Int{\R} \varphi' u \partial_x v dx =  \Int{\R}\varphi' (\partial_t u \partial_x v+u \partial_{xt} v)dx\\
			&=  \Int{\R} \varphi' ( (\partial_x v)^2+u(1-\partial^2_x)^{-1}\partial_x^2 (u+|u|^{p-1}u))dx\\
			&=  \Int{\R} \varphi'  (\partial_x v)^2dx+\Int{\R} \varphi' u(1-\partial^2_x)^{-1} (|u|^{p-1}u)dx\\
			& \quad   +\Int{\R} \varphi'  u(1-\partial^2_x)^{-1}(\partial_{x}^2u-u+u+\partial_{x}^2(|u|^{p-1}u))dx\\
			&=  \Int{\R} \varphi'  (\partial_x v)^2dx -\Int{\R} \varphi'  u^2dx +\Int{\R} \varphi'  u(1-\partial^2_x)^{-1}udx\\
			&  \quad  	+\Int{\R} \varphi'  u(1-\partial^2_x)^{-1}(\partial_{x}^2-1+1)(|u|^{p-1}u)dx\\
			&=  \Int{\R} \varphi'  (\partial_x v)^2dx -\Int{\R} \varphi'  u^2dx +\Int{\R} \varphi'  u(1-\partial^2_x)^{-1}u dx\\
			& \quad    - \Int{\R} \varphi' |u|^{p+1}dx +\Int{\R} \varphi'  u(1-\partial^2_x)^{-1}(|u|^{p-1}u)dx.
		\end{align*}
		The final result arrives after multiplication by $\frac1{2L}$.
	\end{proof}
	
	\bigskip

	\section{Decay in exterior light cones. Proof of Theorem \ref{TH1}}\label{Sect:4}
	
	In this Section we prove Theorem \ref{TH1}. Recall that we have from \eqref{eq:I'}:
	\begin{align}\label{eqn:2p3}
	\dfrac{d}{dt}\I =&~{}  \dfrac{\sigma}{2L}\Int{\R} \varphi'\left( u^2+v^2+(\partial_x v)^2+\frac{2}{p+1} |u|^{p+1}\right) dx \nonumber\\
	&~{} -\frac{1}{L}\Int{\R} \varphi'v\IOp(u+|u|^{p-1}u)dx.
	\end{align}
	In what follows, fix $\sigma\in \R$ such that $|\sigma|>1$. Controlling this last term requires some work. Indeed, we shall need the following definition (see \cite{Dika,KMPP,munoz-kwak} and references therein for more details)
	\begin{defn}[Canonical variable]
	Let $u\in L^2$ be a fixed function. We say that $f$ is canonical variable for $u$ if $f$ uniquely solves the equation
 	 \begin{equation}
		\Op f=u,\ \ f\in H^2(\R). \label{eq:cv_canonical}
	 \end{equation}
	In this case, we denote $f=\IOp u$.
	\end{defn}
	Using  $f$ as canonical variable for $u$, we obtain the following result:
	\begin{lem}\label{lem:passage_canonvar}
	  One has
	   \begin{align}
		   	\Int{\R}\varphi' u^2dx =\Int{\R}\varphi' (f^2+2(\partial_{x} f)^2+(\partial_{x}^2 f)^2)dx-\dfrac{1}{L^2}\Int{\R}\varphi'''f^2dx, \label{eq:u2_canonical}
		   \end{align}
		   and
		   	\begin{align}
			\Int{\R} \varphi'v(1-\partial_{x}^2)^{-1}(u+|u|^{p-1}u) dx =	\Int{\R} \varphi'v(f+(1-\partial_{x}^2)^{-1}|u|^{p-1}u) dx.\label{eq:vIOp_up}
     \end{align}
	\end{lem}
	\begin{proof}
	Computing, 
	\begin{align*}
		\Int{\R}\varphi' u^2dx=\Int{\R}\varphi' (f-\partial_{x}^2 f)^2dx
		=\Int{\R}\varphi' (f^2+(\partial_{x}^2 f)^2-2f\partial_{x}^2 f)dx.
	\end{align*}
	Integrating by parts, we have
	\begin{align*}
		\Int{\R}\varphi'f\partial_{x}^2 fdx
%		=-\Int{\R}(\varphi'f)_{x}\partial_x fdx
%		=-\Int{\R}\left(\frac{1}{L}\varphi''fdx+\varphi'\partial_x f\right)\partial_x fdx
		= -\Int{\R}\varphi'  (\partial_x f)^2dx+\dfrac{1}{2L^2}\Int{\R}\varphi''' f^2dx.
	\end{align*}
	Therefore,
	\begin{align*}
		\Int{\R}\varphi' u^2dx=\Int{\R}\varphi' (f^2+2(\partial_{x} f)^2+(\partial_{x}^2 f)^2)dx-\dfrac{1}{L^2}\Int{\R}\varphi'''f^2dx.
	\end{align*}
This proves \eqref{eq:u2_canonical}. The proof of \eqref{eq:vIOp_up} is direct. 
	\end{proof}
	Using Lemma \ref{lem:passage_canonvar} we  can rewrite \eqref{eqn:2p3} as follows:
	\begin{align*}
		\frac{d}{dt}\mathcal{I}(t)% &~{}  \dfrac{\sigma}{2L}\Int{\R} \varphi'\left( f^2+2(\partial_{x} f)^2+\partial_{x}^2 f^2+v^2+(\partial_x v)^2+\frac{2}{p+1} |u|^{p+1}\right)dx \\
		%&~{}  -\dfrac{\sigma}{2L^3}\Int{\R}\varphi'''f^2 dx -\frac{1}{L}\Int{\R} \varphi'v(1-\partial^{2}_x)^{-1}u^p dx-\frac{1}{L}\Int{\R} \varphi'vf dx\\
		= &~{} \mathcal{Q}(t)+\mathcal{SQ}(t)+\mathcal{PQ}(t),
	\end{align*}
	where
	\begin{align}
		\mathcal{Q}(t)&:=   \dfrac{\sigma}{2L}\Int{\R} \varphi'\left( f^2+2(\partial_{x} f)^2+(\partial_{x}^2 f)^2+v^2+(\partial_x v)^2\right)dx 
		-\frac{1}{L}\Int{\R} \varphi'vf dx, \label{eq:Q}\\
		\mathcal{SQ}(t)&:= -\dfrac{\sigma}{2L^3}\Int{\R}\varphi'''f^2 dx, \label{eq:SQ}\\
		\mathcal{PQ}(t)&:=  \dfrac{\sigma}{L(p+1)}\Int{\R}\varphi' |u|^{p+1}dx-\frac{1}{L}\Int{\R} \varphi'v(1-\partial^{2}_x)^{-1}|u|^{p-1}u dx.\label{eq:PQ}
	\end{align}
	
	Now we are ready to prove a first virial estimate.
	
	\begin{lem}
	Assume $\sigma=-(1+b)<-1$ and $\varphi=\tanh$. Then one has
		\begin{equation}
			\mathcal{Q}(t)\lesssim_{L,b} -\Int{\R} \varphi'\left( u^2+v^2+(\partial_x v)^2\right)dx. \label{eq:sim_norm}
		\end{equation}
		Similarly, assume now $\sigma=1+a>1$ and $\varphi =-\tanh $. Then one has
		\begin{equation}
			\mathcal{Q}(t)\lesssim_{L,a} -\Int{\R} |\varphi'|\left( u^2+v^2+(\partial_x v)^2\right)dx. \label{eq:sim_norm_1}
		\end{equation}
	\end{lem}	
	
	\begin{proof}
	First we prove \eqref{eq:sim_norm}. We concentrate on $\mathcal{Q}(t)$ in \eqref{eq:Q}. Note that, if $\varphi'> 0$, we have
	\begin{align*}
		\left|\Int{\R} \varphi'vf dx\right| 
		%	\leq \left(\Int{\R} \varphi'v^2\right)^{1/2}\left(\Int{\R} \varphi'f^2\right)^{1/2}
		\leq \frac{1}{2}\Int{\R} \varphi'v^2dx+\frac{1}{2}\Int{\R} \varphi'f^2dx.
	\end{align*}
	Consequently, if $b>0$, $\sigma:=-(1+b)<-1$, and $\varphi=\tanh$, we have in \eqref{eq:Q}
		\begin{align*}
			\mathcal{Q}(t)&\leq  \dfrac{\sigma}{2L}\Int{\R} \varphi'\left( f^2+2(\partial_{x} f)^2+(\partial_{x}^2 f)^2+v^2+(\partial_x v)^2\right)dx 
			\\
			& \qquad +\frac{1}{2L}\Int{\R} \varphi'v^2dx+\frac{1}{2L}\Int{\R} \varphi'f^2dx\\
			&=  \dfrac{\sigma+1}{2L}\Int{\R} \varphi'\left( f^2+v^2\right)dx+\dfrac{\sigma}{2L}\Int{\R} \varphi'\left(2(\partial_{x} f)^2+(\partial_{x}^2 f)^2+(\partial_x v)^2\right)dx \\
			&=  \dfrac{-b}{2L}\Int{\R} \varphi'\left( f^2+v^2\right)dx-\dfrac{(1+b)}{2L}\Int{\R} \varphi'\left(2(\partial_{x} f)^2+(\partial_{x}^2 f)^2+(\partial_x v)^2\right)dx.
		\end{align*}
		Now we need the following result about equivalence of norms in terms of $f$ and $u$.
		
		\begin{lem}[\cite{munoz-kwak,KMPP}]\label{lem:equiv_norm}
			Let $f$ be as in \eqref{eq:cv_canonical}. Let $\varphi$  be a smooth, bounded weight function satisfying $\vert \varphi'\vert\leq \lambda \varphi$ for some small but fixed $0<\lambda\ll 1$. Then,  for any $a_1,a_2,a_3>0$, there exist $c_1,C_1>0$, depending of $a_j$ and $\lambda>0$, such that
			\begin{align}
			c_1\Int{\R} \varphi u^2dx\leq \Int{\R} \varphi\left(  a_1 f^2+a_2 (\partial_{x} f)^2+a_3 (\partial_{x}^2 f)^2\right) dx\leq C_1\Int{\R} \varphi u^2 dx.
			\end{align}
		\end{lem}
		Using Lemma \ref{lem:equiv_norm} with $\lambda=L^{-1}\ll 1$, we conclude
		\begin{align}
			\mathcal{Q}(t)\lesssim_{L,b} &~{}  -\Int{\R} \varphi'(f^2+(\partial_{x} f)^2+(\partial_{x}^2 f)^2 +v^2+(\partial_x v)^2) \nonumber\\
			\sim &~{}-\Int{\R} \varphi'\left( u^2+v^2+(\partial_x v)^2\right).\label{I'_decreasing}
		\end{align}
		This proves \eqref{eq:sim_norm}. Now we sketch the proof of \eqref{eq:sim_norm_1}, which is similar to the previous case. Set $\sigma=1+a$, $a>0$. Choosing $\varphi=-\tanh$ it is clear that $\varphi'=-\mbox{sech}^2<0$. From \eqref{eq:Q} we have
		\begin{align*}
			\mathcal{Q}(t)&= -\dfrac{|\sigma|}{2L}\Int{\R} \vert \varphi'\vert \left( f^2+2(\partial_{x} f)^2+(\partial_{x}^2 f)^2+v^2+(\partial_x v)^2\right)dx 
			+\frac{1}{L}\Int{\R} \vert\varphi' \vert vf dx\\
			&\leq  -\dfrac{|\sigma|}{2L}\Int{\R} \vert \varphi'\vert \left( f^2+2(\partial_{x} f)^2+(\partial_{x}^2 f)^2+v^2+(\partial_x v)^2\right)dx \\
			& \qquad +\frac{1}{2L}\Int{\R} \vert \varphi'\vert v^2dx + \frac{1}{2L}\Int{\R} \vert \varphi'\vert f^2dx\\
			&=  -\dfrac{a}{2L}\Int{\R} \vert \varphi'\vert \left( f^2+v^2\right)dx - \dfrac{|\sigma|}{2L}\Int{\R} \vert \varphi'\vert \left(2(\partial_{x} f)^2+(\partial_{x}^2 f)^2+(\partial_x v)^2\right)dx.
		\end{align*}
		Therefore, by Lemma \ref{lem:equiv_norm} again, 
		\begin{align*}
			\mathcal{Q}(t) \lesssim_{L,a}   -\Int{\R} |\varphi'|\left( u^2+v^2+(\partial_x v)^2\right)dx.
			%\label{I'_decreasing_negative}
		\end{align*}
This ends the proof of \eqref{eq:sim_norm_1}.
\end{proof}

Now we consider the two terms in \eqref{eq:SQ} and \eqref{eq:PQ}. First of all, note that in both cases \eqref{eq:sim_norm} and \eqref{eq:sim_norm_1},
\[
\left| \dfrac{\sigma}{2L^3}\Int{\R}\varphi'''f^2 dx \right| \lesssim  \dfrac{1}{L^3}\Int{\R}|\varphi'| f^2 dx \lesssim  \dfrac{1}{L^3}\Int{\R}|\varphi'| u^2 dx.
\]
Therefore, for $L$ large enough,
\begin{equation}\label{intermedio}
\mathcal Q(t) + \mathcal{SQ}(t) \lesssim_{L,a,b}   -\Int{\R} |\varphi'|\left( u^2+v^2+(\partial_x v)^2\right)dx.
\end{equation}
Finally, note that in both cases \eqref{eq:sim_norm} and \eqref{eq:sim_norm_1},
\[
\dfrac{\sigma}{L(p+1)}\Int{\R}\varphi' |u|^{p+1}dx = -\dfrac{|\sigma|}{L(p+1)}\Int{\R}|\varphi'| |u|^{p+1}dx \leq0.
\]
Finally, we deal with the last term in \eqref{eq:PQ}:
\begin{equation}\label{intermedio_2}
\begin{aligned}
\left|\frac{1}{L}\Int{\R} \varphi'v(1-\partial^{2}_x)^{-1}|u|^{p-1}u dx\right| \leq &~{}  \frac{\max\{a, b\}}{8L}\int_\R    |\varphi'| v^2 \\
&~{} + \frac{C}{\max\{a, b\}L} \int_\R   |\varphi'| \left((1-\partial^{2}_x)^{-1}|u|^{p-1}u \right)^2.
\end{aligned}
\end{equation}
The first term on the RHS can be absorbed by $\mathcal Q(t)$ in \eqref{intermedio}.  In what follows, we need the following auxiliary result.
			\begin{lem}[\cite{Dika}, see also \cite{KMPP}]\label{lem:comparison}
			The operator $\IOp$ satisfies the comparison priciple: for any $u,v\in H^1$
			\begin{equation}
			v\leq w\implies  \IOp v\leq \IOp w.
			\end{equation}
		\end{lem}

Now, coming back to \eqref{intermedio_2}, suppose $u\geq 0$. Then $0\leq |u|^{p-1}u \leq \|u\|_{L^\infty}^{p-1} u$, so that using \eqref{smallness0} (this is the only place where we use this hypothesis)
	\[
	0\leq\IOp( |u|^{p-1}u) \leq \|u\|_{L^\infty}^{p-1} \IOp  u  \lesssim_{a,b}  \varepsilon^{p-1} f.
	\]
	(Note that $\varepsilon$ depends on $a,b$.) Consequently, in this region
	\[
	\begin{aligned}
	 |\varphi'|  ((1-\partial_x^2)^{-1}(|u|^{p-1}u))^2  = &~{} |\varphi'|   ((1-\partial_x^2)^{-1}(|u|^{p-1}u))((1-\partial_x^2)^{-1}(|u|^{p-1}u)) \\
	 \lesssim  &~{} \varepsilon^{2(p-1)}  |\varphi'|  f^2.
	\end{aligned}
	\]
	If now $u<0$, just note that 
	\[
	 |\varphi'|   ((1-\partial_x^2)^{-1}(|u|^{p-1}u))^2 =   |\varphi'|   ((1-\partial_x^2)^{-1}(|-u|^{p-1}(-u)))^2,
	\]
	which leads to the previous case. Finally, we conclude that \eqref{auxiliar} is bounded by 
	\[
	\left|\dfrac{d}{dt}\I  \right| \lesssim \frac{\varepsilon^{2(p-1)}}L \Int{\R}  |\varphi'| (v^2 +f^2 )dx.
	\]
	Consequently, for $\varepsilon$ small we obtain
\begin{equation}\label{intermedio_3}
\frac{d}{dt}\mathcal{I}(t) = \mathcal Q(t) + \mathcal{SQ}(t) + \mathcal{PQ}(t) \lesssim_{L,a,b}   -\Int{\R} |\varphi'|\left( u^2+v^2+(\partial_x v)^2\right)dx.
\end{equation}	
Integrating in time, we have proved \eqref{integrability0} in Theorem \ref{TH1}.

	\subsection{End of proof of Theorem \ref{TH1}} Now we conclude the proof of Theorem \ref{TH1}. It only remains to prove \eqref{decay0}. First, we prove decay in the right hand side region, namely $((1+b)t,+\infty)$, $b>0$. Now we choose $\varphi(x)=\frac{1}{2}\left(1+\tanh(x)\right)$, $\sigma=-(1+b)$, $\tilde{\sigma}=-(1+\tilde{b})$ with $b>0$ and $\tilde{b}=b/2$. Consider the modified energy functional, for $t\in[2,t_0]$:
		\begin{align}
			\mathcal{I}_{t_0}(t):= \dfrac{1}{2} \Int{\R} \varphi \left(\dfrac{x+\sigma t_0-\tilde{\sigma}(t_0-t)}{L} \right)\left( u^2+v^2+(\partial_x v)^2+\frac{2}{p+1} |u|^{p+1}\right)dx.
		\end{align}
		Note that $\sigma<\tilde{\sigma}<0$. From Lemma \ref{lem:I'} and proceeding exactly as in \eqref{intermedio_3}, we have
		\begin{align}
			\frac{d }{dt}\mathcal{I}_{t_0}(t) \lesssim_{b,L} -\Int{\R}  \mbox{sech}^2\left(\dfrac{x+\sigma t_0-\tilde{\sigma}(t_0-t)}{L}\right)\left( u^2+v^2+(\partial_x v)^2\right)dx\leq 0
		\end{align}
		what it means that the new functional $\mathcal{I}_{t_0}$ is decreasing in $[2,t_0]$. Therefore, we have
		\begin{align*}
			\int_{2}^{t_0} \dfrac{d}{dt} \mathcal{I}_{t_0}(t) dt= \mathcal{I}_{t_0}(t_0)-\mathcal{I}_{t_0}(2)\leq 0\implies  \ \mathcal{I}_{t_0}(t_0)\leq\mathcal{I}_{t_0}(2).
		\end{align*}
		On the other hand, since $\lim_{x\to -\infty} \varphi (x)=0,$ we have
		\begin{align}
			\limsup_{t\to\infty} \Int{\R} \varphi\left(\dfrac{x-\beta t-\gamma}{L}\right)\left( u^2+v^2+ (\partial_x v)^2\right)(\delta,x)dx =0,
		\end{align}
		for $\beta,\gamma,\delta>0$ fixed. This yields
		\[
		\begin{aligned}
		0 & \leq  \Int{\R} \varphi\left(\dfrac{x-(1+b)t_0}{L}\right)\left( u^2+v^2+ (\partial_x v)^2\right)(t_0,x) dx\\
		& \qquad \qquad \leq  \Int{\R} \varphi\left(\dfrac{x- \frac b2t_0-(2+b)}{L}\right)\left( u^2+v^2+  (\partial_x v)^2\right)(2,x)dx, 
		\end{aligned}
		\]
		which implies,
		\begin{align*}
			\limsup_{t\to\infty} \Int{\R} \varphi\left(\dfrac{x-(1+b)t}{L}\right)\left( u^2+v^2+ (\partial_x v)^2\right)(t,x)dx=0.
		\end{align*}
		In view of \eqref{intermedio_3}, an analogous argument can be applied for the left side, i.e $(-\infty,-(1+a)t)$, but in this case we choose $\varphi(x)=\frac{1}{2}\left(1-\tanh(x)\right)$. 
	 The proof of \eqref{decay0} is complete.

	\bigskip

\section{Decay in compact sets: Proof of Theorem \ref{TH2}}\label{Sect:5}
	
\medskip

Let us find the key virial estimate to understand the dynamics on compact sets in space. Recall that from \eqref{eq:J'}, if $\sigma=0$, we have the identity on $\mathcal J$:
	\begin{equation} \label{eq:Jp}
		\begin{aligned}
			\dfrac{d}{dt}\J  =  & ~ \frac{1}{2L}\Int{\R} \varphi'\left( u^2+\frac{2}{p+1}|u|^{p+1}-v^2- (\partial_x v)^2\right)dx \\
			& ~{} - \frac1L\Int{\R}\varphi'u(1-\partial^2_x)^{-1} (u+|u|^{p-1}u)dx.
		\end{aligned}
	\end{equation}
	Assume that $\varphi'>0$. From \eqref{eq:Jp} and \eqref{eq:Np} we obtain the positivity estimate
	\begin{equation}\label{eq:lyapunov_functional}
		\begin{aligned}
			-\dfrac{d}{dt}(\mathcal{J}(t) +\mathcal{N}(t)) = &~ \frac 1{2L} \Int{\R} \varphi' v^2 dx+ \frac12 \Int{\R} \varphi' u (1-\partial^2_x)^{-1} u dx\\
			& ~{}  +  
				\frac{p-1}{(p+1)L} \Int{\R} \varphi' |u|^{p+1} dx  + 	\frac1{2L}  \Int{\R} \varphi' u(1-\partial^2_x)^{-1} (|u|^{p-1}u)dx .
		\end{aligned}
	\end{equation}
	Note the surprising fact that each term in the RHS above is {\it nonnegative}. 
		
	\begin{lem}\label{lem:posi}
		For any $\varphi$ bounded smooth and such that $\varphi'>0$, and for any $u\in H^1(\R)$,
		\begin{equation}\label{Positividad}
			\Int{\R} \varphi' u(1-\partial^2_x)^{-1} (|u|^{p-1}u)dx \geq 0.
		\end{equation}
	\end{lem}
	
	\begin{rem}
		Note that this result is independent of the size of $u$. Note also that in order to prove this lemma, we need at least $u\in H^{1/2+}(\R)$. Therefore, we see that $(u,v)\in L^2\times H^1$ (the energy space) seems not sufficient for our purposes. 
	\end{rem}

	\begin{proof}[Proof of Lemma \ref{lem:posi}]
		Suppose $\varphi'>0$. If $u\geq 0$, we have $|u|^{p-1}u\geq 0$. From Lemma \ref{lem:comparison} we conclude
		\[
		\Int{\R}  \varphi' u\IOp(|u|^{p-1}u)dx\geq 0.
		\]
		In a similar way, if the case $u\leq 0$ follows. 
		\end{proof}

	From the previous Lemma and \eqref{Positividad} we obtain 
	\begin{equation}\label{eq:J'+N'}
		\begin{split}
			& -\dfrac{d}{dt}(\mathcal{J}(t) +\mathcal{N}(t)) 
			\\
			& \qquad  ~{} \geq  \frac 1{2L} \Int{\R} \varphi' v^2 dx+ \frac1{2L} \Int{\R} \varphi' u (1-\partial^2_x)^{-1} udx +\frac{p-1}{p+1} \Int{\R} \varphi' |u|^{p+1}  dx.
		\end{split}
	\end{equation}
	This last estimate tells us exactly what are the quantities in gIB which integrate in time. As far as we could understand, it was not possible to get integrability in time of the $L^2$ norm of $u$, nor $\partial_x v$. A corollary from this last estimate is the following result.
	
	\begin{cor}\label{cor:int_local}
		Let $(u,v)$ be a global solution of \eqref{IB} in the class $(C\cap L^\infty)(\R,H^1\times H^2)$, with initial data $(u,v)(t=0)=(u_0,v_0)\in H^1\times H^2$. Let $\varphi(x):=\tanh\left(x\right)$ in \eqref{eq:lyapunov_functional}, such that $\varphi'=\sech^2>0$. Then we have the following consequences of \eqref{eq:J'+N'}:
		\begin{enumerate}
		\item Integrability in time:
		\begin{equation}\label{eqn:seq1}
		 \int_{\R}  \Int{\R} \sech^2\left( \frac xL\right)(v^2+u\IOp u + |u|^{p+1})  dxdt \lesssim_{u_0,v_0,L} 1.
		\end{equation}
		\item Sequential decay to zero: there exists $t_n\uparrow \infty$ such that 
		\begin{equation}\label{sucesion}
			 \lim_{n\to \infty} \mathcal I(t_n) =0.
		\end{equation}
		\end{enumerate}
	\end{cor}
	
	\begin{rem}
	Note that the smallness condition on $(u,v)$ is not needed here; only boundedness in time of the $H^1\times H^2$ norm. Also, property \eqref{sucesion} is not trivially obtained from \eqref{eqn:seq1} as in previous works; some additional estimates are needed in order to ensure full decay along a subsequence of the local energy norm present in $\mathcal I(t_n)$.
	\end{rem}
	
	\begin{rem}[About the equivalence of norms for canonical variables]\label{rem:equivalence}
	Note that if $f=\IOp u$, then for $L$ large,
	\[
	\begin{aligned}
	 \Int{\R} \sech^2\left(\frac{x}{L}\right) u\IOp u dx &~= \Int{\R} \sech^2\left(\frac{x}{L}\right) f(f- \partial_x^2 f) dx\\
	 &~= \Int{\R} \sech^2\left(\frac{x}{L}\right) ( f^2 + (\partial_xf)^2)dx  + \frac1L \Int{\R} (\sech^2)'\left(\frac{x}{L}\right) f\partial_x f dx\\
	&~ \gtrsim \frac12\Int{\R} \sech^2\left(\frac{x}{L}\right) ( f^2 + (\partial_xf)^2)dx.
	\end{aligned}
	\]
	Consequently, from \eqref{eqn:seq1},
	\[
	 \Int{\R} \Int{\R} \sech^2\left(\frac{x}{L}\right) ( f^2 + (\partial_xf)^2)dxdt  \lesssim_{u_0,v_0,L} 1.
	\]
	This information will be useful in what follows.
	\end{rem}

	\begin{proof}[Proof of Corollary \ref{cor:int_local}]
Estimate	\eqref{eqn:seq1} is direct from \eqref{eq:J'+N'}. On the other hand, from \eqref{eqn:seq1} we clearly have the existence of an increasing sequence $t_n\uparrow \infty$ such that 
		\[
		\lim_{n\to+\infty} \Int{\R} \sech^2\left( \frac xL\right)(v^2+u\IOp u + |u|^{p+1})(t_n,x)  dxdt =0.
		\]
		From this fact and the $L^\infty$ boundedness in time of $u$ we easily have
			\[
			\lim_{n\to \infty} \Int{\R} \sech^2\left(\frac{x}{L}\right)u^2(t_n,x)  dx =0.
			\]
			Indeed, from Remark \ref{rem:equivalence} we have 
			\[
			\lim_{n\to \infty} \Int{\R} \sech^2\left(\frac{x}{L}\right)(f^2+(\partial_x f)^2)(t_n,x)  dx =0.
			\]
			Hence, using interpolation, and $L\gg1$,
			\[
			\begin{aligned}
			\left\| \sech^2\left( \frac xL\right) \partial_x^2 f(t_n)\right\|_{L^2}^2 \lesssim &~{} \left\| \sech^2\left( \frac xL\right) \partial_x f(t_n)\right\|_{L^2}\left\| \partial_x^3\left( \sech^2\left( \frac xL\right) f(t_n) \right) \right\|_{L^2}\\
			 \lesssim  &~{} \left\|\sech^2\left( \frac xL\right) \partial_x f(t_n)\right\|_{L^2},
			\end{aligned}
			\]
			therefore we obtain the desired result. 
			Finally, again from interpolation, the boundedness of $v(t_n)$ in $H^2$, $L\gg1$, we have the estimate
			\[
			\begin{aligned}
			\left\| \sech^2\left( \frac xL\right) \partial_x v(t_n)\right\|_{L^2}^2 \lesssim &~{} \left\| \sech^2\left( \frac xL\right) v(t_n)\right\|_{L^2}\left\| \partial_x^2\left( \sech^2\left( \frac xL\right) v(t_n) \right) \right\|_{L^2}\\
			 \lesssim  &~{} \left\|\sech^2\left( \frac xL\right) v(t_n)\right\|_{L^2},
			\end{aligned}
			\]
			so $\| \sech^2\left( \frac xL\right) \partial_x v(t_n)\|_{L^2} \to 0$ as $n\to +\infty.$ This proves \eqref{sucesion}.
	\end{proof}

	\subsection{End of proof of Theorem \ref{TH2}} 
	Consider  $\mathcal I(t)$ in \eqref{eq:I} with $\sigma=0$, $\varphi= \sech^2$.  From \eqref{eq:I'} we have
	\[
	\dfrac{d}{dt}\I   = -\frac{1}{L}\Int{\R} \varphi' v(1-\partial_x^2)^{-1}(u+|u|^{p-1}u)dx.
	\]
	Let $g:=\IOp(u+|u|^{p-1}u)$, so that $g= f +\IOp(|u|^{p-1}u)$. We have
	\[
	\dfrac{d}{dt}\I   = -\frac{1}{L}\Int{\R} \varphi' vg dx =-\frac{1}{L}\Int{\R} \varphi' v\left(  f +\IOp(|u|^{p-1}u) \right) dx.
	\]
	Therefore,
	\begin{equation}\label{auxiliar}
		\left|\dfrac{d}{dt}\I  \right| \lesssim \frac1L \Int{\R}\sech^2\left(\frac{x}{L}\right)(v^2 +f^2 + ((1-\partial_x^2)^{-1}(|u|^{p-1}u))^2)dx.
	\end{equation}
	Suppose $u\geq 0$. Then $0\leq |u|^{p-1}u \leq \|u\|_{L^\infty}^{p-1} u$, so that from Lemma \ref{lem:comparison}
	\[
	0\leq\IOp( |u|^{p-1}u) \leq \|u\|_{L^\infty}^{p-1} \IOp  u  \lesssim f.
	\]
	(Here we use the boundedness character of $u$ in $L^\infty$.) Consequently, in this region
	\[
	\begin{aligned}
	 \sech^2\left(\frac{x}{L}\right)  ((1-\partial_x^2)^{-1}(|u|^{p-1}u))^2  = &~{}\sech^2\left(\frac{x}{L}\right)  ((1-\partial_x^2)^{-1}(|u|^{p-1}u))((1-\partial_x^2)^{-1}(|u|^{p-1}u)) \\
	 \lesssim  &~{}  \sech^2\left(\frac{x}{L}\right) f^2.
	\end{aligned}
	\]
	If now $u<0$, just note that 
	\[
	\sech^2\left(\frac{x}{L}\right)  ((1-\partial_x^2)^{-1}(|u|^{p-1}u))^2 =  \sech^2\left(\frac{x}{L}\right)  ((1-\partial_x^2)^{-1}(|-u|^{p-1}(-u)))^2,
	\]
	which leads to the previous case. Finally, we conclude that \eqref{auxiliar} is bounded by 
	\[
	\left|\dfrac{d}{dt}\I  \right| \lesssim \frac1L \Int{\R}\sech^2\left(\frac{x}{L}\right)(v^2 +f^2 )dx.
	\]
	Integrating in $[t,t_n]$, we have (using Remark \ref{rem:equivalence})
	\[
	\left| \I  - \mathcal I(t_n) \right|\lesssim \int_{t}^{t_n} \frac1L \Int{\R}\sech^2\left(\frac{x}{L}\right)(v^2 +f^2 )dxdt.
	\]    
	Sending $n$ to infinity, we have from \eqref{sucesion} that $ \mathcal I(t_n) \to 0$ and
	\[
	\left| \I  \right|\lesssim \int_{t}^{\infty} \frac1L \Int{\R}\sech^2\left(\frac{x}{L}\right)(v^2 +f^2 )dxdt.
	\]    
	Finally, if $t\to \infty$, we conclude. Since for $L^2\times H^1$ data $\mathcal I(t)\gtrsim \|(u,v)(t)\|_{L^2\times H^1}^2$, this proves Theorem \ref{TH2}.

	%\subsection{Proof of Corollary \ref{Corolario_tiempo}} 


\bigskip
\bigskip
	
	\begin{thebibliography}{99}
	
	\bibitem{ACKM} M. A. Alejo, F. Cortez, C. Kwak and C. Mu\~noz, \emph{On the dynamics of zero-speed solutions for Camassa-Holm type equations}, preprint arXiv:1810.09594, to appear in IMRN. 
	
	\bibitem{AM} M. A. Alejo, and C. Mu\~noz, \emph{Almost sharp nonlinear scattering in one-dimensional Born-Infeld equations arising in nonlinear Electrodynamics}, Proc. AMS 146 (2018), no. 5, 2225--2237.
	
	\bibitem{BSS} J. Bona, P. Souganidis, and W. Strauss, \emph{Stability and instability of solitary waves of Korteweg-de Vries type}, Proc. R. Soc. Lond. A 1987 411, 395--412.
	
	\bibitem{Bous} J. Boussinesq, \emph{Th\'eorie des ondes et des remous qui se propagent le long d'{}un canal rectangulaire horizontal, en communiquant au liquide contenu dans ce canal des vitesses sensiblement pareilles de la surface au fond}, J. Math. Pure Appl. (2) 17 (1872), 55--108.

		\bibitem{CC} E. Cerpa, and E. Cr\'epeau, \emph{On the control of the improved Boussinesq equation}, SIAM J. Control Optim., Vol. 56, No. 4, 2018, pp. 3035--3049.  
		
		\bibitem{Chree} Chree, C. \emph{Longitudinal Vibrations of a Corcablar Bar}. The Quarterly Journal of Pure and Applied Mathematics, 21, 287-298,(1886).
		
		\bibitem{GSS} M. Grillakis, J. Shatah,  W. Strauss, \emph{Stability theory of solitary waves in the presence of symmetry}. I. J. Funct. Anal. 74 (1987), no. 1, 160-197.
		
		\bibitem{Notes_linares} F. Linares, \emph{Notes on Boussinesq Equation}, available at \url{http://preprint.impa.br/FullText/Linares__Fri_Dec_23_09_48_59_BRDT_2005/beq.pdf}, 71pp. (2005).
		
		\bibitem{Kishimoto_WP_2012}  N. Kishimoto, \emph{Sharp local well-posedness for the ``good'' Boussinesq equation}. J. Differ. Eqns. 254, 2393--2433 (2013).
		
		\bibitem{wang-chen}  S. Wang, G. Chen. \emph{Small amplitude solutions of the generalized IMBq equation}. J. Math. Anal. Appl. 274(2002) 846--866.
		
		\bibitem{cho-ozawa}  Y. Cho, T. Ozawa, \emph{Remarks on Modified Improved Boussinesq Equations in One Space Dimension}, Proceedings: Mathematical, Physical and Engineering Sciences, Vol. 462, No. 2071 (2006), pp. 1949--1963.
		
		\bibitem{munoz-kwak} C. Kwak, C. Mu\~noz, \emph{Extended Decay Properties for Generalized BBM Equations}, to appear in Fields Institute Comm. preprint 2018.
		
		\bibitem{KM} C. Kwak, and C. Mu\~noz, \emph{Asymptotic dynamics for the small data weakly dispersive one-dimensional hamiltonian ABCD system}, preprint arXiv:1902.00454. 
		
		\bibitem{KMPP} C. Kwak, C. Mu\~noz, F. Poblete, and J.C. Pozo, \emph{The scattering problem for the abcd Boussinesq system in the energy space}, to appear in J. Math. Pures Appl.
		
		\bibitem{Dika} K. El Dika, \emph{Smoothing effect of the generalized BBM equation for locelized solutions moving to the right}. Sidcrete and Contin. Dyn. syst. 12(2005), no 5, 973--982.
		
		
		
		\bibitem{KMM1} M. Kowalczyk, Y. Martel, and C Mu\~noz, \emph{Kink dynamics in the $\phi^4$ model: asymptotic stability for odd perturbations in the energy space}, J. AMS 30 (2017), 769--798. 
		
		\bibitem{KMM2} M. Kowalczyk, Y. Martel, and C Mu\~noz, \emph{Nonexistence of small, odd breathers for a class of nonlinear wave equations}, L. Math. Phys. May 2017, Vol. 107, 5, pp. 921--931.
		
		\bibitem{KMM_3} M. Kowalczyk, Y. Martel, and C Mu\~noz, \emph{Soliton dynamics for the 1D NLKG equation with symmetry and in the absence of internal modes},  preprint arXiv:1903.12460.
		
		\bibitem{Liu} Y. Liu, \emph{Existence and Blow up of Solutions of a Nonlinear Pochhammer-Chree Equation}, Indiana University Mathematics Journal Vol. 45, No. 3 (Fall, 1996), pp. 797--816.
		
		\bibitem{MM1} Y. Martel, and F. Merle, \emph{Asymptotic stability of solitons for subcritical generalized {K}d{V} equations}, Arch. Ration. Mech. Anal. \textbf{157} (2001), no.~3, 219--254.
  
  		\bibitem{MM2} Y. Martel, and F. Merle,  \emph{Asymptotic stability of solitons for subcritical  gKdV equations revisited}, Nonlinearity, \textbf{18} (2005), no.~1, 55--80. 

		\bibitem{Mizu1} T. Mizumachi, \emph{Stability of line solitons for the KP-II equation in $\R^2$}, Mem. Amer. Math. Soc. 238 (2015), no. 1125, vii+95 pp.
		
		\bibitem{Mizu2} T. Mizumachi, \emph{Stability of line solitons for the KP-II equation in $\R^2$. II.} Proc. Roy. Soc. Edinburgh Sect. A 148 (2018), no. 1, 149--198.
				
		\bibitem{MPP} C. Mu\~noz, F. Poblete, and J.C. Pozo, \emph{Scattering in the energy space for Boussinesq equations}, Comm. Math. Phys. 361 (2018), no. 1, 127--141.

		\bibitem{Poch} L. Pochhammer, \emph{Ueber die Fortpflanzungsgeschwindigkeiten kleiner Schwingungen in einem unbegrenzten isotropen Kreiscylinder}. Journal f\"ur die Reine und Angewandte Mathematik 81 (1876): 324-336.
		
		\bibitem{PW} R. Pego, M.  Weinstein, \emph{Eigenvalues, and instabilities of solitary waves}. Philos. Trans. Roy. Soc. London Ser. A 340 (1992), no. 1656, 47--94.

		\bibitem{PW_2} R. Pego, M.  Weinstein, \emph{Convective Linear Stability of Solitary Waves for Boussinesq Equations}, Studies in Applied Mathematics, 99, pp. 311-375 (1997).

		
		\bibitem{Smereka}  P. Smereka, \emph{A remark on the solitary wave stability for a Boussinesq equation}. Nonlinear dispersive wave systems (Orlando, FL, 1991), 255-263, World Sci. Publ., River Edge, NJ, 1992.
		
		\bibitem{Whitham} Whitham, \emph{Linear and nonlinear waves}, Pure and Applied Mathematics, John Wiley, 1974, 636pp.
		
	\end{thebibliography}
\end{document}